\documentclass[12pt]{article}
\usepackage{amssymb,amsmath,amsthm,epsf}
\def\bbc{{\mathbb C}}  \def\bbr{{\mathbb R}}

\def\ep{\epsilon}

\def\lgrarrow{\mathrel{\hbox to 30pt{\rightarrowfill}}}
\def\eop{\hfill$\Box$}
\newtheorem{thm}{Theorem}
\newtheorem{lem}{Lemma}

\begin{document}
\title{Measures for orthogonal polynomials  with unbounded recurrence coefficients}

\author{A. I. Aptekarev\\ Keldysh Institute for Applied Mathematics\\ Russiand Academy of Sciences \\ Musskaya pl. 4, 12507\\ Moscow, Russia \\and\\ J.\ S.\ Geronimo\\School of Mathematics \\ Georgia Institute of Technology,\\ Atlanta, Ga, 30332,\ USA \\ and \\ Tao Aoqing Visiting Professor\\ Jilin University\\ Changchun, China}

\maketitle
\baselineskip=22pt

\begin{abstract}
 Systems of orthogonal polynomials whose recurrence coefficients tend to
infinity are considered. A summability condition is imposed on the coefficients
and the consequences for the measure of orthogonality are discussed. Also 
discussed are  asymptotics for the polynomials.



\end{abstract}

\section{Introduction}
Let  $\{ p_{n}(x)\}^\infty_{n=-1}$ be a system of polynomials 
satisfying
the recurrence relations
\begin{equation}\label{eq0.1}
a_{n+1} p_{n+1}(x)+b_{n} p_{n}(x)+ a_{n} p_{n-1}(x)=x
p_{n}(x), \qquad  p_{0}(x)=1,\qquad p_{-1}(x)=0,
\end{equation}
with $a_{n+1}>0$ and $b_n$ real for $n=0,1,\ldots$.
By a theorem of Farvard these polynomials are orthonormal with respect to some positive probability measure supported on the real line. 
There has been much work done on the asymptotics or spectral properties of polynomials whose recurrence coefficients are
unbounded \cite{d, dp, jn, ns} and here we are interested in the problem of constructing
the orthogonality measure given the coefficients in the recurrence formula.
This problem for the bounded case has been an area of ongoning intense investigation due to its connection to the discrete Schr\"odinger equation \cite{dk, den, geca, gi, ger, ks, l, mn, mnt, nev, Va, vg}  and to the connection between the continuum limits of the recurrence relations with varying recurrence coefficients and  discrete integrable systems \cite{DM, Arno-Walter}. However with regard to the construction of the measure of orthogonality almost all the results are for cases
with bounded recurrence coefficients. Unfortunately many of the techniques developed for
the bounded case cannot be applied to the unbounded case. Careful analysis of the work of  M\'at\'e-Nevai \cite{mn}, M\'at\'e-Nevai-Totik \cite{mnt}, and VanAssche-Geronimo \cite{vg}  for coefficients that are of bounded variation i.e.
$$
\sum_{n=1}^{\infty}|a_{n+1}-a_n|+|b_{n}-b_{n-1}|<\infty
$$
with limits $a_n\to \frac{1}{2}$ and $b_n\to0$ reveals
that it is possible to modify these techniques so that they apply to certain
cases when the recurrence coefficients tend to infinity. For bounded recurrence coefficients obeying the above criteria  the absolutely continuous part of the orthogonality measure is given by
\begin{equation}\label{eq0.2}
du_{ac}=\frac2\pi\frac{\sqrt{1-x^2}dx}{|\xi(x)|^2
\prod^\infty_{n=1}|\zeta_n|^2(x)},
\end{equation}
 where
\begin{equation*}
\zeta_n(x):=\frac{x-b_n+\sqrt{(x-b_n)^2-4a^2_n}} {2a_n}
\end{equation*}
is the mapping function $z = \zeta_n (x)$ of
$\overline{\bbc}\backslash(\alpha_n,\,\,\beta_n)$ on $\overline{\bbc} \backslash\{|z|\le 1\}$
normalized as $\zeta (x) / x > 0$ when $x \rightarrow \infty$
, and for $\xi(x)$ we have the expressions
\begin{equation*}
\xi(x)=1+\sum^\infty_{k=1}\left\{
\frac1{\zeta_k\left(x\right)}-
  \frac{a_{k}/a_{k+1}}{\zeta_{k+1}\left(x\right)}\right\}
 \frac{p_{k-1}(x)}{\prod^{k}_{j=1}\zeta_j(x)}\,,
\end{equation*}
or
\begin{equation*}
\xi(x)=\lim_{n\to\infty} \frac{p_{n}(x)-\frac{b_{n}}{b_{n+1}}\,
p_{n-1}(x)/\zeta_{n+1}\left(x\right)}{\prod^n_{j=1} \zeta_{j}\left(x\right)}\,,
\end{equation*}
An extension of this formula was used in \cite{AGVa} to show the connection between the limit of varying recurrence coefficients and their corresponding orthogonality measure and an analog of this formula for the unbounded case is necessary in order to extend the above results to the unbounded case. This is done in the next section.

\section{Unbounded Coefficients}

We now consider the case when $a_n>0, n>0$ and $b_n$ real for $n\ge0$  tend in
magnitude to infinity when $n$ tends to infinity i.e.
$$\lim_{n\to\infty} a_n=\infty\quad\mbox{and}\quad
\lim_{n\to\infty} |b_n|=\infty.$$
Let $\rho (z)=z+\sqrt{z^2-1}$, $t_{1,n} = \sqrt{\frac{a_n}{a_{n+1}}}
\rho\left(\frac{x-b_n}{2\sqrt{a_na_{n+1}}}\right)$, and $t_{2,n} =
\sqrt{\frac{a_n}{a_{n+1}}}
\rho\left(\frac{x-b_n}{2\sqrt{a_na_{n+1}}}\right)^{-1}$. We also set
$t_{i,j}=1,\ i=1,2,\ j=-1,0$. Since $\rho$ maps the exterior of $[-1,1]$ to the exterior of the unit circle these functions are nonzero.  Let
\begin{equation}\label{eq1}
\phi_n(x)=p_n(x)-t_{2,n}p_{n-1}(x),\end{equation}
and
$$
\phi^1_n(z)=p^1_n(x)-t_{2,n}p^1_{n-1}(x)
$$
where $p^1_j\ j=1,2..$ are the second kind polynomials of degree $j-1$ (we take $p^1_0=0$). Using the recurrence formula for the orthogonal polynomials and (\ref{eq1}) yields
\begin{equation}\label{recurphi}
\phi_{n+1}-t_{1,n}\phi_n=(t_{2,n}-t_{2,n+1})p_n.
\end{equation}
We consider the systems of recurrence coefficients given by
$\left\{a^{n_0}_{n},b^{n_0}_{n-1}\right\}^\infty_{n=1}$,
$n_0=1,2,\dots$ where
$$a^{n_0}_i=\begin{cases}a_i & i<n_0\\ a_{n_0} & i\ge n_0\end{cases}$$
and
$$b^{n_0}_i=\begin{cases}b_i & i<n_0\\ b_{n_0} & i\ge n_0\end{cases}$$
Let $p^{n_0}_i(x)$ be the polynomials constructed satisfying
\begin{equation}\label{eq2}
a^{n_0}_{i+1}p^{n_0}_{i+1}(x)+b^{n_0}_i p^{n_0}_i(x)+
a^{n_0}_ip^{n_0}_{i-1}(x)=xp^{n_0}_i(x)\end{equation}
$p^{n_0}_{-1}(x)=0$, $p^{n_0}_0(x)=1$. Set $\phi^{n_0}_{n_0}=p^{n_0}_{n_0}-\rho(\frac{x-b_{n_0}}{2a_{n_0}})^{-1}p^{n_0}_{n_0-1}$ and $\phi^{n_0,1}_{n_0}=p^{n_0,1}_{n_0}-\rho(\frac{x-b_{n_0}}{2a_{n_0}})^{-1}p^{n_0,1}_{n_0-1}$.
Then it is well known 
Geronimo and Case \cite{geca}, Geronimo and Iliev \cite{gi}, Geronimus \cite{ger}, Nevai \cite{nev} that the above polynomials are orthonormal
with respect to a measure $\mu^{n_0}$ where $\mu^{n_0}$ is given by
$$d\mu^{n_0}(x)=\begin{cases}\displaystyle
\frac{\sqrt{4a^2_{n_0}-(x-b_{n_0})^2}}
{2|\phi^{n_0}_{n_0}(x)|^2a^2_{n_0}\pi} dx &  x\in\left[b_{n_0}-2a_{n_0},
b_{n_0}+2a_{n_0}\right]\\
\displaystyle\sum^{m(n_0)}_{i=1}\mu^{n_0}_i\delta (x-x^{n_0}_i)dx & x \text{ elsewhere},
\end{cases}$$

where
$$
\mu^{n_0}_i=\frac{\phi_{n_0}^{n_0,1}(x^{n_0}_i)}{\frac{d\phi_{n_0}^{n_0}(x^{n_0}_i)}{dx}}.
$$
Here $x_i$ is a zero of $\phi^{n_0}_{n_0}$ for
$x\in\bbr\setminus\left[a_{n_0}-2b_{n_0}, a_{n_0}+2b_{n_0}\right]$ all of
which are real and  simple. The following Lemma is found in \cite{gi} and has a simple proof.
\begin{lem} \label{lem1}
Suppose that the moment problem is determinate then $\mu^{n_0}(x)$
converge weakly to $\mu(x)$ where $\mu(x)$ is the orthogonality measure
associated with the polynomials $\{p_i(x)\}^\infty_{i=-1}$.
\end{lem}

{\bf Proof}\quad It follows from equation \eqref{eq2} that the polynomials
$p^{n_0}_i(x)=p_i(x)$ for $0\le i\le n_0$ so that the moments of the original system $s_i$ and the perturbed system $s_i^{n_0}$ agree for
$0\le i\le 2n_0$.  The result now follows from the fact that the moment
problem is determinate.\eop

We now consider in more detail the properties of $\mu$. Set $\bbc^+=\{x\in\bbc: \Im x\ge 0\}$

\begin{thm}\label{thm1}
Suppose $\lim_{n\to\infty}a_n=\infty$, $\lim_{n\to\infty}\frac{a_n}{a_{n+1}}
=1$, $\lim_{n\to\infty}\left(\frac{b_n}{2a_{n+1}}\right)=d\ne \pm1$,
$\sum^\infty_{n=1}\left|\frac{a_n}{a_{n+1}}-\frac{a_{n+1}}{a_{n+2}}\right|
<\infty$, $\sum^\infty_{n=1}\left|\frac{1}{a_{n+1}}-\frac{1}{a_{n}}\right|
<\infty$, and $\sum^\infty_{n=0}\left|\frac{b_n}{a_{n+1}}-
\frac{b_{n+1}}{a_{n+2}}\right|<\infty$.  Then
\begin{equation}
\lim_{n\to\infty}
\frac{\phi_n(x)}{\prod^{n-1}_{i=1}t_{1,i}}=g(x).
\end{equation}
where the convergence is
uniform on compact subsets of $\bbc^+$. Thus $g$ is continuous for $\Im x\ge 0$ and analytic for $\Im x>0$. The same holds for $\lim_{n\to\infty}
\frac{\phi^1_n(x)}{\prod^{n-1}_{i=2}t_{1,i}}=g^1(x)$
\end{thm}

\begin{proof}
Let $K$ be a compact subset of $\bbc^+$.  It follows
from the first three hypotheses that for every $x$ in $K$ there is
an $k_0$ such that for $k\ge k_0$,
\begin{equation}\label{eq3}
\frac{a_k}{a_{k+1}}\sqrt{\left(\frac{x-b_{k}}{2a_{k+1}}\right)^2-\frac{a_k}{a_{k+1}}}\ne 0
\ne\frac{a_k}{a_{k+1}}\sqrt{\left(\frac{x-b_{k-1}}{2a_{k-1}}\right)^2-\frac{a_k}{a_{k-1}}}
+\sqrt{\left(\frac{x-b_{k+1}}{2a_{k+2}}\right)^2-\frac{a_{k+1}}{a_{k+2}}}
\end{equation}
Consequently
$$\left|t_{1,k+1}-\frac{a_k}{a_{k+1}}\, t^{-1}_{2,k-1}\right|\le O
(\ep_k+\ep_{k-1}+\ep_{k-2})$$
where
$$\ep_i=\left|\frac{1}{a_{i+1}}-\frac{1}{a_{i+2}}\right|+\left|\frac{a_i}{a_{i+1}}-\frac{a_{i+1}}{a_{i+2}}\right| +
\left|\frac{b_i}{a_{i+1}}-\frac{b_{i+1}}{a_{i+2}}\right|,$$
for $i>0$.

Let
\begin{equation}\label{eq4}
G(k,m)=\begin{cases}0 & k\ge m\\
\sum^m_{i=k+1}\left(\prod^m_{j=i+1} t_{1,j}\right)\left(
\prod^{i-1}_{j=k+1}t_{2,j}\right) & k<m,\end{cases}\end{equation}
and $\hat G(k,m)=\prod^m_{i=k}t^{-1}_{1,i}G(k,m)$ for $k>0$.  Then from
\cite[p.~228]{vg}.
\begin{align}
&\left(\frac{a_k}{a_{k+1}}-\frac{a_{k+1}}{a_{k+2}}\frac{t_{1,k}}{t_{2,k+1}}
\right)\hat G(k,m)=-1+\frac{a_{k-1}}{a_k}\frac{1}{t_{1,m-1}t_{1,m}} R(k,m)
\nonumber\\
&\qquad + \sum^{m-1}_{i=k+1} t_{2,i-1}\left\{t_{1,i+1}-
\frac{a_i}{a_{i+1}}\, t^{-1}_{2,i-1}\right\}
R(k,i)\hat G(i,m)\end{align}
where $R(k,m)=\prod^{m-1}_{j=k+1}t_{2,j-1}/t_{1,j}$. Since
$\sqrt{\frac{a_{j+1}}{a_j}}|t_{2,j}|\le 1\le\sqrt{\frac{a_{j+1}}{a_j}}|t_{1,j}|$ and
$$\lim_{k\to\infty}\left(\frac{a_k}{a_{k+1}}-\frac{a_{k+1}}{a_{k+2}}
\frac{t_{1,k}}{t_{2,k+1}}\right)=1-\rho(-d)^2$$ we
see that for all $x\in K$ there is a constant $c$ and a $k_0$ such that
$|R(k,m)|<c$ for $k\ge k_0$.
Also for $k_0$ large enough there is an $r>0$ such that
$\left|\frac{a_k}{a_{k+1}}-\frac{a_{k+1}}{a_{k+2}}\frac{t_{1,k}}{t_{2,k+1}}
\right|>r$ for all $k\ge k_0$.  Thus Picard iteration gives
$|\hat G(k,m)|\le c_1\exp\left\{D\sum^m_{i=k}\ep_i\right\}$ for $k\ge k_0$.

Since $G(k,m)$ satisfies the relation
$$
\frac{a_{k+1}}{a_{k+2}}
G(k+1,m)+G(k-1,m)-(t_{1,k+1}+t_{2,k})G(k,m)=\delta_{k,m}
$$
\cite[eq.~(3.2)]{vg}
after multiplication by $\prod_{i=k}^m t^{-1}_{1,k}$ this relation allows
the computation of $\hat G(i,m), i=-1,0$ so the above bound may be extended
by  induction to $-1\le k< k_0$.     

The polynomials $\{p_n(x)\}^\infty_{m=0}$ satisfy the relation
\cite[eq.~(3.5)]{vg}
\begin{equation}\label{eq7}
\hat p_m(x)=\hat G(-1,m)+\sum^{m-1}_{k=0}t_{1,k} (t_{1,k}-t_{1,k+1})\hat
p_k(x)\hat G(k,m),
\end{equation}
where $\hat\phi _n(z)=\left(\prod^{n-1}_{i=1} t^{-1}_{1,i}\right)\phi _n(z)$. Using the above bound on $G(k,m)$ $k=-1,\dots, m-1$ and Picard
iteration yields (\cite[eq.~(3.10)]{vg})
\begin{equation}\label{eq8}
|\hat p_m(x)|\le A\exp\left\{B\sum^m_{k=1}|t_{1,k}-t_{1,k+1}|\right\}
\le A\exp\left\{ \tilde B\sum^m_{k=1} \ep_k\right\},\end{equation}
where the constant $A$, $B$ and $\tilde B$ may be chosen uniformly for
$x\in K$.
Finally from equation~\eqref{recurphi} we find
\begin{equation}\label{hatphone}
\hat\phi_{n+1}(z)-\hat \phi _n(z)=(t_{2,n}-t_{2,n+1})\hat p_n(z),
\end{equation}
where $\hat\phi _n(z)=\left(\prod^{n-1}_{i=1} t^{-1}_{1,i}\right)\phi _n(z)$.
Thus $\hat\phi_n(z)-\hat\phi_m(z)=\sum^{n-1}_{i=m}(t_{2,i}
-t_{2,i+1})\hat p_j(z)$.  The bound \eqref{eq8} can now be used to obtain 
\begin{equation}\label{hatp}
|\hat\phi_n(z)-\hat\phi_m(z)|\le O\left(\sum^n_{i=m}\ep_i\right).
\end{equation}
Thus $\{\hat\phi_n(z)\}$ is a Cauchy sequence in every compact subset
$K\subset\bbc^+$ which gives the uniform convergence. Since each $t_{i,n},\
i=1,2\ n=1,2..$ is continuous for $\Im x\ge0$ and analytic for $\Im x>0$
the continuity and analyticity properties of $g$ follow. An analogous
argument gives the result for $\hat\phi^1_n$.
\end{proof}

\begin{lem}\label{lem2} Suppose the hypothesis of  {\rm Theorem~\ref{thm1}}
  hold with $|d|>1$. Then for each compact set $K\subset\bbc$ there exits
  and $N$ such that $\prod_{i=1}^N t_{1,i} g(x)$ is analytic for $x\in
  K$. The same is true for $\prod_{i=2}^N t_{1,i} g^1(x)$.
\end{lem}

\begin{proof}
If $d>1$ then for each compact set $K\subset\bbc$ there exists
an $N$ such for every $x\in K,\ \Re(x)<b_N-2 a_N$. Thus for $n\ge N,\
t_{i,n}$ are analytic for $x\in K$ so the result follows from
Theorem~\ref{thm1}. A similar argument can be used for $d<-1$. The result
for  $\prod_{i=2}^N t_{1,i} g^1(x)$ follows as above.
\end{proof}

\begin{thm}\label{thm2}
Suppose the hypotheses of {\rm Theorem~\ref{thm1}} hold and $d\ne\pm1$.
Suppose that the moment problem associated with the
recurrence coefficients $\{a_{n+1},b_{n}\}^{\infty}_{n=0}$ is determinate. If $|d|<1$ then
$\mu(x)$ is absolutely continuous and
$$\mu' (x)=\frac{1}{a_1\pi}\frac{\sqrt{1-d^2}}{|g(x)|^2\prod^\infty_{n=1}|\tilde t_{1,n}|^2}
\, .\ x\in\bbr$$
If $|d|>1$ then $\mu$ is purely discrete with,
$$
d\mu=\sum_{i=0}^{\infty} \mu_i\delta(x-x_i)dx,
$$
where
$$
\mu_i=\frac{g^1(x_i)}{t_{1,1}(x_i)\frac {dg(x_i)}{dx}},
$$
and $x_i$ are the real zeros of $g(x)$.
Here $\tilde t_{1,n}=\rho\left(\frac{x-b_n}{2\sqrt{a_na_{n+1}}}\right)$ and
$t_{1,1}(x_i)=\lim_{y\to0}t_{1,1}(x_i+iy)$.
\end{thm}

\begin{proof}
From the definition of $\phi_{n_0}^{n_0}$
we see that $\phi_{n_0}=\phi
^{n_0}_{n_0}+(\rho(\frac{x-b_{n_0}}{2a_{n_0}})^{-1}-t_{2,n_0})p_{n_0-1}$. Thus
Theorem 1 implies that $\hat\phi_{n_0}^{n_0}=\prod_{i=1}^{n_0-1} t_{1,i}\phi_{n_0}^{n_0}$ converges uniformly on compact subsets of $\bbc^+$ to $g$. For $|d|<1$ it follows that for each compact
subset $K\subset R$ there exists an $N$, such that for all $n>N$,
$|\tilde t_{1,n}|=1$.
Since for large enough $n_0$, $K\subset [b(n_0)-2a(n_0),\ b(n_0)+
2a(n_0)]$ we see that if $n_0$ is also greater than $N$
\begin{align*}
d\mu^{n_0}_{ac}(x)&=\frac{
\sqrt{1-\left(\frac{x-b_{n_0}}{2a_{n_0}}\right)^2}}
{|\hat \phi^{n_0}_{n_0}(x)|^2 a_{n_0}\pi\prod^{n_0-1}_{i=1}
|\tilde t_{1,i}|^2\frac{a_1}{a_{n_0}}}\, dx\\
&=\frac{1}{\pi a_1}\frac{\sqrt{1-\left(\frac{x-b_{n_0}}{2a_{n_0}}\right)^2}}
{|\hat \phi^{n_0}_{n_0}(x)|^2\prod^{n_0-1}_{i=1}
|\tilde t_{1,i}|^2}\, dx.\end{align*}
Now from Theorem~(\ref{thm1}) and Lemma~(\ref{lem1}b)
$$\lim_{n_0\to\infty}\frac{1}{\pi a_1}\frac{
\sqrt{1-\left(\frac{x-b_{n_0}}{2a_{n_0}}\right)^2}}
{|\hat\phi^{n_0}_{n_0}(x)|^2\prod^{n_0-1}_{i=1}|\tilde t_{1,i}|^2}=
\frac{1}{\pi a_1}\frac{\sqrt{1-d^2}}{g(x)\prod^\infty_{i=1}
|\tilde t_{1,i}|^2}$$
where the convergence is uniform on compact subsets of $R$.  Note
that $\prod^\infty_{i=1}|\tilde t_{1,i}|^2=\prod^N_{i=1}
|\tilde t_{1,i}|^2$ for $x\in K$ which gives the result for $|d|<1$. We now
consider $|d|>1$. Since the zeros of $\phi_{n_0}^{n_0}$ are real and simple
it follows from  Theorem~\ref{thm1},  Lemma~\ref{lem2}, and Rouche's theorem
that the zeros of $g$ are real and simple. The weak convergence of
$\mu^{n_0}$ shows that $\mu$ is purely discrete and  from Theorem~\ref{thm1} and Lemma~\ref{lem2} we find 
$$
\lim_{n_0\to\infty}\mu^{n_0}_i=\lim_{n_0\to\infty}\frac{\phi_{n_0}^{1,n_0}(x^{n_0}_i)}{\frac{d\phi_{n_0}^{n_0}(x^{n_0}_i)}{dx}}=\lim_{n_0\to\infty}\frac{\hat\phi^{1,n_0}_{n_0}(x^{n_0}_i)}{t_{1,1}(x^{n_0}_i)\frac{d\hat\phi^{n_0}_{n_0}(x^{n_0}_i)}{d x}}= \frac{g^1(x_i)}{t_{1,1}(x_i)\frac{d g(x_i)}{ dx}},
$$
which gives  the result.
\end{proof}

The above results allow us to obtain  asymptotics for the orthogonal
polynomials in an analogous manner to some of those obtained in
M\'at\'e-Nevai-Totik \cite{mnt}.
\begin{thm}\label{asymone}
Suppose the hypotheses of Theorem~\ref{thm1} hold, If $|d|<1$ and $K$ is a
compact subset of the real line then there exist an $N$ so that
\begin{align}\label{asymdgt1rl}
\sqrt{a_{n+1}}\sqrt{\big(1-(\frac{x-b_{n+1}}{2\sqrt{a_{n+2}a_{n+1}}})^2\big)\mu'(x)}p_n(x)&=\sqrt{\frac{\sqrt{1-d^2}}{\pi}}\sin\big(\sum_{k=1}^{n}\arg{t_{1,k}+\arg{g(x)}}\big)\\&+O(\sum_{k=n+1}^{\infty} \ep_k)
\end{align}
for all $n\ge N$ uniformly for $x\in K$. If $|d|>1$ or $|d|<1$ and $K$ is a compact subset of $\Im x>0$ then
\begin{equation}\label{asydlto}
 \frac{(p_n(x)-t_{2,n}p_{n-1}(x))}{\prod_{i=1}^{n-1}t_{1,i}}=g(x)+O(\sum_{k=n}^{\infty} \ep_k),
\end{equation}
uniformly for $x\in K$. For $|d|>1$ this is also true for $K\subset\bbc\setminus\text{supp}(\mu)$
\end{thm}
\begin{proof}
 Multiplication of equation~\eqref{eq1} by $\frac{1}{\prod_{i=1}^{n-1}t_{1,i}}$ then using equation~\eqref{hatp} and Theorem~\ref{thm1}
yields
\begin{equation}\label{asympt}
\frac{p_n(x)-t_{2,n}p_{n-1}(x)}{\prod_{i=1}^{n-1}t_{1,i}}=g(x)+O(\sum_{k=n}^{\infty} \ep_k).
\end{equation}
This gives the result for  $K\in\Im x>0$. For $|d|>1$ and $K\subset\bbc\setminus\text{supp}(\mu)$ the result follows from Lemma~\ref{lem2}.
If $K$ is a compact subset of the real line and $|d|<1$ then for large enough $n,\ K\subset[b_n-2a_n,b_n+2a_n]$ so taking the imaginary part of the above equation yields
\begin{align*}
\sqrt{\frac{a_n}{a_1}}\sqrt{1-(\frac{x-b_n}{2\sqrt{a_{n+1} a_n}})^2}p_{n-1}(x)&=|g(x)|\prod_{i=1}^{n-1}|\tilde t_{1,i}|\sin\big(\sum_{k=1}^{n-1}\arg{t_{1,k}+\arg{g(x)}}\big)\\&+O(\sum_{k=n}^{\infty} \ep_k).
\end{align*}
The use of Theorem~\ref{thm2} implies equation~\eqref{asymdgt1rl}. 
\end{proof}

\section{Acknowledgements} A.I. Aptekarev was supported by grant RScF-14-21-00025. He would also like to thank the School of Mathematics at GT for its hospitality (Fall 2004) when part of this work was completed.
J.S. Geronimo was partially supported by Simons Foundation Grant 210169. He would also like to thank the JLU-GT Joint Institute for Theoretical Sciences for their hospitiality where some of the work was carried out.

\end{document}